\documentclass[11pt]{amsart}
\usepackage{amssymb}
\usepackage{amsmath}
\usepackage[all, cmtip, line]{xy}
\usepackage{tikz-cd}
\usepackage{pbox}
\usepackage{array}
\usepackage[utf8]{inputenc}
\usepackage{color}

\newtheorem{thm}{Theorem}[section]

\newtheorem{cor}[thm]{Corollary}
\newtheorem{lem}[thm]{Lemma}
\newtheorem{rem}[thm]{Remark}

\newtheorem{example}[thm]{Example}

\newcommand{\B}{{\mathcal B}}
\newcommand{\C}{{\mathcal C}}
\newcommand{\D}{{\mathcal D}}

\newcommand{\E}{{\mathcal E}}

\newcommand{\I}{{\mathcal I}}

\newcommand{\Irr}{\operatorname{Irr}}
\newcommand\FPdim{\operatorname{FPdim}}

\newcommand\vect{\operatorname{Vec}}

\newcommand\Pic{\operatorname{G}}

\begin{document}

\title[Braided $\mathbb{Z}_q$-extensions of pointed fusion categories]{Braided $\mathbb{Z}_q$-extensions of pointed fusion categories}
\author{Jingcheng Dong}
\address{College of Engineering, Nanjing Agricultural University, Nanjing 210031, China}
\email{dongjc@njau.edu.cn}



\date{\today}

\begin{abstract}
We classify braided $\mathbb{Z}_q$-extensions of pointed fusion categories, where $q$ is a prime number. As an application, we classify modular categories of Frobenius-Perron dimension $q^3$.
\end{abstract}
 \maketitle

%%% ----------------------------------------------------------------------

%%% ----------------------------------------------------------------------
%\tableofcontents

\section{Introduction}\label{intro}
Let $k$ be an algebraically closed field of characteristic $0$. By definition, a fusion category is a $k$-linear semisimple rigid tensor category with finitely many isomorphism classes of simple objects, finite-dimensional spaces of morphisms, and such that the unit object $\textbf{1}$ is simple. We refer the reader to \cite{etingof2005fusion} for main notions and basic results on fusion categories.

Let $\C$ be a fusion category and let $G$ be a finite group with identity element $0$. A $G$-grading on $\C$ is a decomposition $\C=\oplus_{g\in G}\C_g$ as a direct sum of full Abelian subcategory such that the dual functor $*$ sends $\C_g$ into $\C_{g^{-1}}$ and the tensor product $\otimes:\C\times\C\to \C$ maps $\C_g\times\C_h$ into $\C_{gh}$.

The grading $\C=\oplus_{g\in G}\C_g$ is called faithful if $\C_g\neq 0$ for all $g\in G$. If $\C$ has a faithful $G$-grading and $\C_0=\D$ then $\C$ is called a $G$-extension of $\D$.

Group extensions of fusion categories paly important roles in classifying fusion categories and have been intensively studied by several authors \cite{gelaki2008nilpotent,etingof2005fusion,etingof2011weakly}. An important class of group extensions is the  $\mathbb{Z}_q$-extensions of pointed fusion categories, where $q$ is a prime number. By a pointed fusion category we mean a fusion category whose simple objects are all invertible. Several typical examples of $\mathbb{Z}_q$-extensions of pointed fusion categories are recalled in Section \ref{sec2}.

The main work of this paper is to classify braided $\mathbb{Z}_q$-extensions of pointed fusion categories. Let $\C$ be a braided $\mathbb{Z}_q$-extension of a pointed fusion category. Suppose that $\C$ is not pointed. We prove that $\C$ is equivalent to a Deligne tensor product $\B\boxtimes\E$, where $\B$ is a pointed fusion category, $\E$ is a fusion category of $q$ power dimension. In particular,  $\E$ is of type $(1,m;\alpha,n)$ for some positive integers $m,n$ and  $\alpha^2$ is a power of $q$. Suppose further that $\C$ is modular. Then we prove that $\C$ is equivalent to a Deligne tensor product $\B\boxtimes \I$, where $\B$ is a pointed modular category and $\I$ is an Ising category. Finally, we apply these results to modular categories of Frobenius-Perron (FP) dimension $q^3$, and prove that this class of modular categories are equivalent to $\B\boxtimes\I$, where $\B$ is a pointed modular category of FP dimension $2$, $\I$ is an Ising category.

\section{Preliminaries and Examples}\label{sec2}
Let $\C=\oplus_{g\in G}\C_g$ be a faithful grading on $\C$. Then the FP dimensions of $\C_g$ are all equal \cite[Proposition 8.20]{etingof2005fusion}, and hence we have $\FPdim(\C)=|G|\FPdim(\C_0)$, where we denote by $\FPdim(\C)$ the FP dimension of $\C$.

We denote by $\Irr(\C)$ the set of non-isomorphic simple objects of $\C$, and by $\Irr_{\alpha}(\C)$ the set of non-isomorphic simple objects of FP dimension $\alpha$. The adjoint subcategory $\C_{ad}$ is the full tensor subcategory of $\C$ generated by simple objects in $X\otimes X^*$, for all $X\in \Irr(\C)$. The rank of $\C$ is the cardinality of the set $\Irr(\C)$.

Every fusion category $\C$ has a unique faithful grading $\C=\oplus_{g\in \mathcal{U}(\C)}\C_g$ such that $\C_0=\C_{ad}$. This grading is called the universal grading of $\C$ and the group $\mathcal{U}(\C)$ is called the universal grading group of $\C$. This grading is universal because any faithful grading  $\C=\oplus_{g\in G}\C_g$ comes from a surjective group homomorphism $\mathcal{U}(\C)\to G$. This universal property implies the following result:

\begin{lem}\label{lem21}
Let $\C=\oplus_{g\in G}\C_g$ be a faithful grading on $\C$. Then $\C_{ad}\subseteq \C_0$.
\end{lem}

A fusion category $\C$ is called integral if $\FPdim(X)$ is an integer for all objects $X$ in $\C$, where $\FPdim(X)$ denotes the FP dimension of $X$. A fusion category $\C$ is called weakly integral if $\FPdim(\C)$ is an integer. Let $\C$ be a $G$-extension of a pointed fusion category $\D$. Then $\FPdim(\C)=|G|\FPdim(\D)$, and hence $\C$ is weakly integral since $\FPdim(\D)$ is an integer.

Let $\C_{pt}$ denote the fusion subcategory generated by all invertible simple objects of $\C$. Then $\C_{pt}$ is the largest pointed fusion subcategory of $\C$. All non-isomorphic invertible objects of $\C$ form a group with multiplication given by tensor product. We denote this group by $\Pic(G)$. The group $\Pic(G)$ acts on the set $\Irr(\C)$ by left tensor multiplication, and this action preserves FP dimension. For $X\in \Irr(\C)$, we use $\Pic[X]$ to denote the stabilizer of $X$ under this action.

We now discuss some examples of $\mathbb{Z}_q$-extensions of pointed fusion categories.

\begin{example}\label{example22}
{\rm (Generalized Tambara-Yamagami fusion categories) Let $\C$ be a fusion category. If $\C$ is not pointed and $X\otimes Y$ is a direct sum of invertible objects, for all non-invertible simple objects $X,Y\in\C$, then $\C$ is a generalized Tambara-Yamagami fusion category.}

{\rm Generalized Tambara-Yamagami fusion categories were classified in \cite{liptrap2010generalized}, up to equivalence of tensor categories, and then were further studied in \cite{natale2013faithful}. By \cite{natale2013faithful}, there exists a normal subgroup $N$ of $G(\C)$ such that $\Pic[X]=N$, for all non-invertible simple objects $X$. This implies that $\FPdim(X)=\sqrt{|N|}$ for all non-invertible simple objects $X$. The rank of $\C$ is $[G(\C):N](1+|N|)$, and hence $\FPdim(\C)=2|G(\C)|$. This implies that $\C$ is a $\mathbb{Z}_2$-extension of a pointed fusion category generated by $G(\C)$.}
\end{example}

\begin{example}\label{example23}
{\rm (Tambara-Yamagami fusion categories) Let $\C$ be a generalized Tambara-Yamagami fusion category. If $N=G(\C)$ then the rank of $\C_{\textbf{1}}$ is $1$. Then we can write $\Irr(\C)=G(\C)\cup\{X\}$, where $X$ is the unique non-invertible simple object of $\C$. In this case, $\C$ is called a Tambara-Yamagami fusion category. This class of fusion categories were classified in \cite{Tambara1998692}.}
\end{example}

\begin{example}\label{example24}
{\rm (Ising fusion category) Let $\C$ be a generalized Tambara-Yamagami fusion category.  If the order of $G(\C)$ is $2$ and $N=G(\C)$ then $\C$ is called an Ising fusion category. It is well known that any Ising category admits a structure of braided category. This class of fusion categories were classified in \cite[Appendix B]{drinfeld2010braided}.}
\end{example}

\begin{example} \label{example25}
{\rm A fusion category $\C$ is called nilpotent if there is a sequence of fusion categories $\vect_k=\C_0\subseteq\C_1\subseteq\cdots\subseteq\C_n=\C$  and a sequence of finite groups $G_1,\cdots,G_n$ such that $\C_i$ is obtained from $\C_{i-1}$ by a $G_i$-extension, for all $1\leq i\leq n$. If the groups $G_1,\cdots,G_n$ can be chosen to be cyclic of prime order then $\C$ is called cyclically nilpotent.}

{\rm Let $\C$ be a cyclically nilpotent fusion category, and let  $\C_0,\C_1,\cdots,\C_n$,  $G_1,\cdots,G_n$ be the corresponding fusion subcategories and finite groups. Since $\C_0$ is the trivial fusion category and $\C_1$ is a $\mathbb{Z}_p$-extension of $\C_0$ for some prime number $p$, $\C_1$ is a pointed fusion category. It follows that $\C_2$ is a $\mathbb{Z}_q$-extension of a pointed fusion category ($\C_1$) for some prime number $q$.}
\end{example}

\begin{lem}\label{lem23}
Let $q$ be a prime number and let $\C=\oplus_{g\in \mathbb{Z}_q}\C_g$ be a faithful grading of $\C$. Assume that the trivial component $\C_0$ is pointed. Then

(1)\, The adjoint subcategory $\C_{ad}$ is pointed;

(2)\, $\C$ is pointed, or $\C_0=\C_{pt}$ is the largest pointed fusion subcategory of $\C$.
\end{lem}

\begin{proof}
(1)\, By Lemma \ref{lem21}, $\C_{ad}$ is contained in $\C_0$. Hence, $\C_{ad}$ is pointed.

(2)\, Since $\C_{pt}$ is the unique largest pointed fusion subcategory of $\C$, $\C_{pt}$ contains $\C_0$ as a fusion subcategory. This fact shows that $\FPdim(\C_0)$ divides $\FPdim(\C_{pt})$ \cite[Proposition 8.15]{etingof2005fusion}. On the other hand, $\FPdim(\C)=q\FPdim(\C_0)$. These facts imply that $\FPdim(\C_{pt})=\FPdim(\C_0)$ or $q\FPdim(\C_0)$. The first case means that $\C_0=\C_{pt}$, and the second case means that $\C=\C_{pt}$ is pointed.
\end{proof}

\begin{lem}\label{lem24}
Let $\C$ be a $\mathbb{Z}_q$-extension of a pointed fusion category. Assume that $\C$ is not pointed. Then $\C$ is a generalized Tambara-Yamagami fusion category if and only if $q=2$.
\end{lem}
\begin{proof}
Let $\C=\oplus_{g\in \mathbb{Z}_q}\C_g$ be the $\mathbb{Z}_q$-extension. If $q=2$ then Lemma \ref{lem23} shows that $\C_0=\C_{pt}$ is the largest pointed fusion subcategory of $\C$, and hence all non-invertible simple objects are contained in $\C_1$. Let $X,Y$ be non-invertible simple objects of $\C$. Then $X\otimes Y$ is contained in $\C_0$, which means that $X\otimes Y$ is a direct sum of invertible simple objects. Hence, $\C$ is a generalized Tambara-Yamagami fusion category.

Conversely, if $\C$ is a generalized Tambara-Yamagami fusion category then $\FPdim(\C)=2|G(\C)|$  by Example \ref{example22}. On the other hand, $\FPdim(\C)=q\FPdim(\C_0)=q|G(\C)|$ by Lemma \ref{lem23}. Hence $q=2$ as claimed.
\end{proof}

\section{Main results}
Recall that a fusion category $\C$ is braided if it has a natural isomorphism $c_{X,Y}:X\otimes Y\to Y\otimes X$, for all $X,Y$ in $\C$,  which satisfies the hexagon axioms \cite{kassel1995quantum}.

Let $1 = d_0 < d_1< \cdots < d_s$ be positive real numbers, and let $n_0,n_1,\cdots,n_s$ be positive integers. A fusion category is said of type $(d_0,n_0; d_1,n_1;\cdots;d_s,n_s)$ if $n_i$ is the number of the non-isomorphic simple objects of Frobenius-Perron dimension $d_i$, for all $0\leq i\leq s$.

\begin{thm}\label{thm31}
Let $\C$ be a braided $\mathbb{Z}_q$-extension of a pointed fusion category. Suppose that $\C$ is not pointed. Then $\C\cong \B\boxtimes\E$, where $\B$ is a pointed fusion category, $\E$ is a fusion category of $q$ power dimension. In particular,  $\E$ is of type $(1,m;\alpha,n)$ for some positive integers $m,n$ and  $\alpha^2$ is a power of $q$.
\end{thm}
\begin{proof}
Let $X_1,X_2,\cdots,X_s$ be a list of all non-isomorphic simple objects of $\C$ such that $1<\FPdim(X_1)\leq\cdots\leq\FPdim(X_s)$. We may say that $\FPdim(X_1)=\alpha$ for some positive real number $\alpha$. By Lemma \ref{lem21} and Lemma \ref{lem23}, $\C_{ad}\subseteq\C_0=\C_{pt}$. So  $X_1\otimes X_1^*$ is contained in $\C_{pt}$, and hence the stabilizer $G[X_1]$ of $X_1$ under the action of $G[\C]$ is of order $\alpha^2$. Let $\D$ be the fusion subcategory generated by simple objects in $G[X_1]$. It is a pointed fusion subcategory of $\C$ with FP dimension $\alpha^2$.

Let $\D^{co}$ be the commutator of $\D$ in $\C$; that is, $\D^{co}$ is the fusion subcategory of $\C$ generated by all simple objects $X$ of $\C$ such that $X\otimes X^*$ is contained in $\D$ \cite{gelaki2008nilpotent}. Clearly, all invertible objects and $X_1$ are contained in $\D^{co}$, which means that $\FPdim(\D^{co})\geq\FPdim(\C_{pt})+\alpha^2$. On the other hand, $\FPdim(\C_{pt})$ divides $\FPdim(\D^{co})$ and $\FPdim(\D^{co})$ divides $\FPdim(\C)=q\FPdim(\C_{pt})$. This implies that $\FPdim(\D^{co})=\FPdim(\C)$. Hence, we must have that $\D^{co}=\C$.

By \cite[Lemma 4.15]{gelaki2008nilpotent}, $\C_{ad}=(\D^{co})_{ad}\subseteq\D$. On the other hand, $\C_{ad}$ has FP dimension at least $\alpha^2$ since $X_1\otimes X_1^*$ is contained in it. It follows that $\C_{ad}=\D$. Since $\C_{ad}$ has FP dimension $\alpha^2$, $\C$ can not have simple objects with FP dimension greater than $\alpha$. In other words, the FP dimensions of simple objects of $\C$ can only be $1$ or $\alpha$.

Since $\C$ is braided and nilpotent,  \cite[Theorem 1.1]{drinfeld2007group} shows that we have a decomposition $\C\cong\boxtimes_{p_i}\C_{p_i}$, where $\C_{p_i}$ is a fusion subcategory of prime power dimension. The simple object $X$ of $\C$ has the form $X\cong \otimes_{p_i}X_{p_i}$, where $X_{p_i}$ is a simple object of $\C_{p_i}$.  So if there exist $\C_{p_i}$ and $\C_{p_k}$ such that they both contain non-invertible simple objects, then the number of distinct FP dimensions of simple objects of $\C$ is at least $3$. This contradicts the results obtained above. Therefore, there is only one subcategory in the decomposition of $\C$ such that it contains non-invertible simple objects. So $\C$ has the decomposition $\C\cong \B\boxtimes\E$, where $\B$ is a pointed fusion category, $\E$ is a fusion category of a prime power dimension. In particular,  $\E$ is of type $(1,m;\alpha,n)$ for some positive integers $m,n$. We shall prove that $\FPdim(\E)$ is a power of $q$.

Assume that $\FPdim(\E)=p^a$ for some prime number $p$, and hence we may assume that $\FPdim(\E_{pt})=p^i$ for some $0\leq i\leq a-1$. Lemma \ref{lem23}(2) and our assumption show that $\C_0=\C_{pt}$, and hence $\C_0=\B\boxtimes \E_{pt}$.

On the one hand, $\FPdim(\C)=q\FPdim(\C_0)$ since $\C$ is a $\mathbb{Z}_q$-extension of $\C_0$. On the other hand, $\FPdim(\C)=\FPdim(\B)\FPdim(\E)$ since $\C$ has the decomposition $\B\boxtimes\E$. Hence, we have
\begin{align*}
\FPdim(\B)\FPdim(\E)&=q\FPdim(\C_0);\\
\FPdim(\B)p^a&=q\FPdim(\B)p^i;\\
p^a&=q\cdot p^i.
\end{align*}
This means that $p=q$ and $i=a-1$. So $\FPdim(\E)$ is a power of $q$. Finally, \cite[Theorem 2.11]{etingof2011weakly} shows that $\alpha^2$ is a power of $q$.
\end{proof}

\begin{cor}\label{cor32}
Let $\C$ be a braided $\mathbb{Z}_q$-extension of a pointed fusion category. Then

(1)\, If $\C$ is not integral then $\C$ is a generalized Tambara-Yamagami fusion category. In this case $\C\cong \B\boxtimes\E$, where $\B$ is a pointed fusion category, $\E$ has FP dimension $2^k$ for some $k$.

(2)\, If $q>2$ then $\C$ is integral.
\end{cor}

\begin{proof}
(1)\,Let $\C=\oplus_{g\in \mathcal{U}(\C)}\C_g$ be the universal grading of $\C$. Then $\alpha^2=\FPdim(\C_g)$ for all $g\in \mathcal{U}(\C)$, by the proof of Theorem \ref{thm31}. It follows that every component $\C_g$ either contains only one simple object with FP dimension $\alpha$, or contains $\alpha^2$ non-isomorphic invertible simple objects. This hence implies that, for any $X, Y\in \Irr_{\alpha}(\C)$, $X\otimes Y$ is either a direct sum of $\alpha^2$ invertible simple objects, or a direct sum of $\alpha$ copies of a simple object with FP dimension $\alpha$. Note that the later case holds true if and only if $\C$ is integral. Hence, if $\C$ is not integral then $X\otimes Y$ is a direct sum of invertible simple objects, so $\C$ is a generalized Tambara-Yamagami fusion category.

Let $\E$ be the braided fusion category of $q$ power dimension in Theorem \ref{thm31}, and let $\FPdim(\E)=q^k$ for some $k$. Since $\C$ is not integral then $\E$ is not integral. This happens only if $q=2$ \cite[Corollary 3.11]{gelaki2008nilpotent}.

(2)\, Suppose on the contrary that $\C$ is not integral. Then $\C$ is a generalized Tambara-Yamagami fusion category by Part (1), and hence $q=2$. This is a contradiction.
\end{proof}

Recall that a fusion category is called group-theoretical if it is Morita equivalent to a pointed fusion category.

 \begin{rem}\label{rem33}
 If $\C$ is integral then $\C$ is a group-theoretical fusion category by \cite[Theorem 6.10]{drinfeld2007group}, since $\C$ is braided and nilpotent of nilpotency class $2$.
 \end{rem}

 Let $\C$ be a braided fusion category, and $\D$ be a fusion subcategory of $\C$. The M\"{u}ger centralizer of $\D$ in $\C$ is the fusion subcategory
 $$\D'=\{X\in \C|c_{Y,X}c_{X,Y}=id_{X\otimes Y},\,\mbox{for all\,\,} Y\in \D\}.$$

 The M\"{u}ger center $\mathcal{Z}_2(\C)$ of $\C$ is the M\"{u}ger centralizer of $\C$ itself. A braided fusion category $\C$ is non-degenerate if its M\"{u}ger center $\mathcal{Z}_2(\C)$ is trivial. A braided fusion category is premodular if  it has a spherical structure. By \cite[Proposition 8.23, 8.24]{etingof2005fusion}, any weakly integral fusion category is premodular. Hence a weakly integral braided fusion category is modular if and only if it is non-degenerate. In particular, if $\C$ is a braided $G$-extension of a pointed fusion category then $\C$ is modular if and only if it is non-degenerate.

\begin{thm}\label{thm34}
Let $\C$ be a $\mathbb{Z}_q$-extension of a pointed fusion category. Assume in addition that $\C$ is modular. Then $\C$ fits into one of the following classes:

(1)\,$\C$ is pointed;

(2)\,$\C\cong \B\boxtimes \I$, where $\B$ is a pointed modular category and $\I$ is an Ising category.
\end{thm}

\begin{proof}
We may assume that $\C$ is not pointed. Since $\C$ is modular, we have $\mathcal{U}(\C)\cong G(\C)$ \cite[Theorem 6.2]{gelaki2008nilpotent}. It follows that the order of $\mathcal{U}(\C)$ is equal to $|G(\C)|=\FPdim(\C_{pt})$. Let $\C=\oplus_{g\in\mathbb{Z}_q}\C_g$ be the $\mathbb{Z}_q$-extension. By Lemma \ref{lem23},
$$\FPdim(\C_{pt})=\FPdim(\C_0),$$
and hence
 $$\FPdim(\C)=q\FPdim(\C_0)=q\FPdim(\C_{pt}).$$
On the other hand, $$\FPdim(\C)=|\mathcal{U}(\C)|\FPdim(\C_g)=\FPdim(\C_{pt})\FPdim(\C_g),$$
for any $g\in\mathcal{U}(\C)$. Hence, every component of the universal grading has FP dimension $q$. In particular, $\C_{ad}$ has FP dimension $q$.

Let $X$ be a non-invertible simple object of $\C$. Then the proof of Theorem \ref{thm31} shows that $\FPdim(\C_{ad})=\FPdim(X)^2$, which means that $\FPdim(X)=\sqrt{q}$. Since $q$ is a prime number, $\sqrt{q}$ is not an integer, and hence $\C$ is not integral. By Corollary \ref{cor32}, $\C$ is a generalized Tambara-Yamagami fusion category. Therefore $\C\cong \B\boxtimes \I$ as described above, by \cite[Theorem 5.4]{natale2013faithful}.
\end{proof}

\begin{cor}\label{cor45}
Let $q$ be a prime number and let $\C$ be a modular category of FP dimension $q^3$. Then

(1)\, $\C$ is pointed, or

(2)\, $\C\cong \B\boxtimes\I$, where $\B$ is a pointed modular category of FP dimension $2$, $\I$ is an Ising category.
\end{cor}
\begin{proof}
\cite[Theorem 8.28]{etingof2005fusion} shows that $\C$ is a $\mathbb{Z}_q$-extension of a fusion category $\D$, where $\D$ has FP dimension $q^2$. By \cite[Proposition 8.32]{etingof2005fusion}, $\D$ is either pointed or an Ising fusion category.

If $\D$ is pointed then Theorem \ref{thm34} shows that $\C$ is either pointed, or equivalent to $\B\boxtimes \I$, where $\B$ is a pointed modular category of FP dimension $2$, $\I$ is an Ising category.

If $\D$ is  an Ising fusion category then $\D$ is a modular category by \cite[Corollary B.12]{drinfeld2010braided}. By M\"{u}ger's Theorem \cite[Theorem 4.2]{muger2003structure}, $\C$ is equivalent to $\D\boxtimes \D'$, where $\D'$ is the M\"{u}ger centralizer of $\D$. Again by M\"{u}ger's Theorem, $\D'$ is a modular category since $\C$ is modular. Finally, $\D'$ is a pointed fusion category by  \cite[Corollary 8.30]{etingof2005fusion} since its FP dimension is $2$.
\end{proof}

%------ACKNOWLEDGEMENTS--------------------

\section{Acknowledgements}
This paper was written during a visit to University of Southern California. The author would like to thank Susan Montgomery and the Department of Mathematics of University of Southern California for their outstanding hospitality. The author also thanks Henry Tucker for useful discussion. This work is partially supported by the Fundamental Research Funds for the Central Universities (KYZ201564), the Natural Science Foundation of China (11201231) and the Qing Lan Project.

%-----BIBLIOGRAPHY--------------------------

%\bibliographystyle{G:/references/bibstyles/elsart-num-sort}

%\bibliography{G:/references/dongrefs}

\end{document}